\documentclass{amsart}

\usepackage[all]{xypic}
\usepackage[centertags]{amsmath}
\usepackage{amsfonts}
\usepackage{amscd}
\usepackage{amssymb}
\usepackage{amsthm}
\usepackage{newlfont}
\usepackage{amsxtra}
 \newtheorem{thm}{Theorem}[section]
 \newtheorem{cor}[thm]{Corollary}
 
 \newtheorem{lem}[thm]{Lemma}
 \newtheorem{prop}[thm]{Proposition}
 \theoremstyle{definition}
 \newtheorem{defn}[thm]{Definition}
 \theoremstyle{remark}
 \newtheorem{rem}[thm]{Remark}
 \theoremstyle{remark}
 \newtheorem{example}[thm]{Example}
 \theoremstyle{definition}
 \newtheorem{notn}[thm]{Notation}
 \numberwithin{equation}{section}

 \newcommand{\Ver}{\mathrm{Ver}}
 
 \newcommand{\Par}{\mathrm{Par}}

 \newcommand{\ord}{\mathrm{ord}}

 \newcommand{\Vol}{\mathrm{Vol}}
 
 \newcommand{\GL}{\mathrm{GL}}
 \newcommand{\PGL}{\mathrm{PGL}}

 \newcommand{\Type}{\mathrm{Type}}

 \newcommand{\inv}{\mathrm{inv}}
 \renewcommand{\mod}{\mathrm{mod}}

 \newcommand{\fE}{\mathfrak E}

 \newcommand{\cO}{\mathcal{O}}

 \newcommand{\cB}{\mathcal{B}}

 \renewcommand{\cD}{\mathcal{D}}

 \newcommand{\cF}{\mathcal{F}}

 \renewcommand{\cL}{\mathcal{L}}

 \newcommand{\R}{\mathbb{R}}
 \newcommand{\C}{\mathbb{C}}
 
 \newcommand{\F}{\mathbb{F}}
 \newcommand{\M}{\mathbb{M}}
 \newcommand{\Q}{\mathbb{Q}}
 \newcommand{\Z}{\mathbb{Z}}

 \newcommand{\bp}{\mathbf{p}}

 \newcommand{\bs}{\setminus}

 \newcommand{\Fi}{F_\infty}

 \newcommand{\bg}{\overline{\Gamma}}
 \newcommand{\G}{\Gamma}
 
 \newcommand{\La}{\Lambda}

 \newcommand{\qbinom}[2]{\left[\begin{matrix} #1\\ #2\end{matrix}\right]_q}

\begin{document}

\title[On finite arithmetic simplicial complexes]
{On finite arithmetic simplicial complexes}

\author{Mihran Papikian}

\address{Department of Mathematics, Pennsylvania State University, University Park, PA 16802}

\email{papikian@math.psu.edu}

\thanks{The author was supported in part by NSF grant DMS-0801208.}



\begin{abstract}
We compute the Euler-Poincar\'e characteristic of quotients of the
Bruhat-Tits building of $\PGL(n)$ under the action of arithmetic
groups arising from central division algebras over rational function
fields of positive characteristic. We use this result to determine
the structure of the quotient simplicial complex in certain cases.
\end{abstract}


\maketitle


\section{Introduction} The purpose of this article is to generalize
to higher dimensions the genus formula for modular curves of
$\cD$-elliptic sheaves proven in \cite{PapGenus}. The proof of this
genus formula given in \cite{PapGenus} relies in part on the
arithmetic of $\cD$-elliptic sheaves. In this paper we avoid the use
of the theory of $\cD$-elliptic sheaves, and work exclusively within
the arithmetic of central division algebras over $F:=\F_q(T)$; here
$\F_q$ denotes the finite field with $q$ elements and $T$ is an
indeterminate.

Denote by $A=\F_q[T]$ the subring of $F$ formed by the polynomials
in $T$. For $0\neq f\in A$, let $\deg(f)$ be the degree of $f$ as a
polynomial in $T$, and put $\deg(0)=+\infty$. For $f/g\in F$ with
$f,g\in A$, let $\deg(f/g):=\deg(f)-\deg(g)$. Then $-\deg$ defines a
valuation on $F$; the corresponding place is denoted by $\infty$.
Let $\Fi$ be the completion of $F$ at $\infty$. Let $n\geq 2$ and
let $D$ be a central division algebra over $F$ of dimension $n^2$.
Assume $D\otimes_F F_\infty$ is isomorphic to the matrix algebra
$\M_n(\Fi)$. Fix a maximal $A$-order $\La$ in $D$ and denote by
$\G:=\La^\times$ its group of units. Let $\cB$ be the Bruhat-Tits
building of $\PGL_n(\Fi)$. The group $\G$ acts on $\cB$, and the
quotient $\G\bs\cB$ is a finite simplicial complex. The genus
formula in \cite{PapGenus} is equivalent to a formula for the
Euler-Poincar\'e characteristic $\chi(\G\bs\cB)$ when $n=2$. In this
paper we generalize this formula to arbitrary prime $n$. We also
determine the possible stabilizers in $\G$ of simplices of $\cB$,
and the number of $\G$-orbits of simplices with a given stabilizer
(Theorem \ref{thmChi}). These results are sufficient for determining
$\G\bs\cB$ in the case when $D$ is ramified at exactly two rational
places (Theorem \ref{LastThm}): in this special case, $\G\bs\cB$ is
a sort of a ``multi-layered'' $(n-1)$-simplex. I am not aware of
other instances where the quotient $\G\bs\cB$ is explicitly
determined when $n\geq 3$ (but see \cite{Soule} for $\GL_n(A)\bs
\cB$).

The quotients $\G\bs\cB$ play an important role in many arithmetic
problems, e.g. the theory of automorphic forms over function fields.
For congruence subgroups $\G$ of $\GL_2(A)$, the quotient graphs
$\G\bs\cB$ have been extensively studied by Gekeler and others in
relation to the theory of Drinfeld modular forms, cf.
\cite{GekelerKF}, \cite{GN}, \cite{GR}.

The proof of Theorem \ref{thmChi} uses two key ingredients. One is
Serre's theory relating Euler-Poincar\'e characteristics of discrete
subgroups of non-archimedean Lie groups to measures, and the other
is Eichler's formula for the number of non-equivalent optimal
embeddings of an order into a central simple algebra.


\subsection*{Notation}

The following notation is fixed throughout the paper.

\vspace{0.1in}

$|F|$ = the set of places of $F$.

$F_x$ = the completion of $F$ at $x\in |F|$.

$\cO_x = \{z\in F_x\ |\ \ord_x(z)\geq 0\}$ = the ring of integers of
$F_x$.

$\pi_x = \{z\in F_x\ |\ \ord_x(z)> 0\}$ = the maximal ideal of
$\cO_x$.

$\F_x=\cO_x/\pi_x$.

$q_x=\#\F_x$.

$\deg(x)=[\F_x:\F_q]$.

\vspace{0.1in}

Since the place $\infty$ plays a special role in our arguments, to
simplify the notation we put $K:=F_\infty$, $\cO:=\cO_\infty$,
$\pi:=\pi_\infty$, $\ord:=\ord_\infty$.

\section{Bruhat-Tits building of $\PGL_n(K)$}

\subsection{Simplicial complexes} By a simplicial complex we mean a
usual abstract simplicial complex, cf. \cite[p. 15]{Munkres}, except
that we allow for two distinct simplices of the same positive
dimension to have the same sets of vertices. More precisely, a
\textit{simplicial complex} $X$ is a collection of non-empty sets
$$
S_0(X),\ S_1(X), \dots,\ S_n(X), \quad n\leq \infty
$$
where each $s\in S_i(X)$ is a subset of $S_0(X)$ of cardinality
$i+1$, and each subset of $s$ of cardinality $j+1$, $0\leq j\leq i$,
is in $S_j(X)$. We call $s\in S_i(X)$ an \textit{$i$-simplex} and
each nonempty subset of $s$ a \textit{face} of $s$. The
\textit{vertices} of the simplex $s$ are the one-point elements of
the subset $s\subset S_0(X)$. With this terminology, the elements of
$S_0(X)$ are called the vertices of $X$. We will denote
$\Ver(X)=S_0(X)$. If $n$ is finite, we call it the
\textit{dimension} of $X$. $X$ is \textit{finite} if it is a finite
set. The \textit{generalized $m$-th degree} of $s\in X$, denoted
$\deg_X^m(s)$, is the number of elements of $S_m(X)$ having $s$ as a
face.

Let $X$ be finite of dimension $n$. The \textit{Euler-Poincar\'e
characteristic} of $X$ is
$$
\chi(X):=\sum_{i=0}^n (-1)^i\# S_i(X).
$$
One can define the cohomology groups of $X$ (with $\Q$-coefficients)
$H^\ast(X, \Q)$ in the usual manner, cf. \cite{Munkres}. Then
$$
\chi(X)=\sum_{i=0}^{n}(-1)^i\dim_\Q H^i(X, \Q).
$$

We say that a group $G$ \textit{acts} on $X$ if $G$ acts on the set
of simplices of $X$ and this action satisfies the following
condition: if $s\in S_i(X)$ has vertices $\{v_0, \dots, v_i\}$, then
$gs\in S_i(X)$ has vertices $\{gv_0, \dots, gv_i\}$, $g\in G$. We
single out an extra condition on the action of $G$:
\begin{equation}\label{eq-star}
\text{If $gs=s$ for a simplex $s$, then $g$ fixes all the vertices
of $s$}.
\end{equation}
If $G$ acts on $X$ and satisfies (\ref{eq-star}), then there is a
natural quotient simplicial complex $Y:=G\bs X$ such that
$S_i(Y)=S_i(X)/G$ for all $i$. For $s\in X$, denote $O_s=G\cdot s$
the orbit of $s$ under the action of $G$, i.e., $O_s=\{gs\ |\ g\in
G\}$. The action of $G$ decomposes $S_i(X)$ into a disjoint union of
orbits $O_s$, and the set of these orbits is in bijection with
$S_i(Y)$. Denote
$$
G_s = \{g\in G\ |\ gs=s\}
$$
the stabilizer of $s$. Let $\tilde{s}\in Y$ and $s$ be a preimage of
$\tilde{s}$ in $X$. We define $O_{\tilde{s}}=O_s$ and $\#
G_{\tilde{s}}=\# G_s$; the second definition makes sense since the
elements in the orbit $O_s$ have isomorphic stabilizers: $G_{gs}=g
G_s g^{-1}$.

\begin{lem}\label{lem2.1} Assume $G$ is finite and $X$ is finite of dimension $n$.
With previous notation,
$$
\chi(Y)=\frac{\chi(X)}{\# G}+\sum_{i=0}^n (-1)^i\sum_{\tilde{s}\in
S_i(Y)} \left(1-\frac{1}{\# G_{\tilde{s}}}\right).
$$
\end{lem}
\begin{proof}  Since the orbits are disjoint, we have
\begin{align*}
\# S_i(X) &= \sum_{\tilde{s}\in S_i(Y)} \# O_{\tilde{s}} =
\sum_{\tilde{s}\in S_i(Y)} \frac{\# G}{\# G_{\tilde{s}}}\\ &=\#
G\#S_i(Y)+\sum_{\tilde{s}\in S_i(Y)} \# G\left(\frac{1}{\#
G_{\tilde{s}}}-1\right).
\end{align*}
Now take the alternating sums of both sides over $0\leq i\leq n$.
\end{proof}


\subsection{The building}\label{Sec2.2} A \textit{lattice} in $K^n$ is
any finitely generated $\cO$-submodule of $K^n$ which contains a
basis of this vector space; such a module is free of rank $n$. If
$x\in K^\times$ and $L$ is a lattice in $K^n$, then $xL$ is also a
lattice in $K^n$. Thus the group $K^\times$ acts on the set of
lattices $\cL$. Denote the quotient $\cL/K^\times$ by $\bar{\cL}$.

Define a simplicial complex $\cB$ as follows. Let
$\Ver(\cB)=\bar{\cL}$. A finite subset of $\bar{\cL}$ is an
$i$-simplex of $\cB$ if one can represent its elements by lattices
$L_0,\dots, L_i$ such that
\begin{equation}\label{eq-simB}
L_0\supsetneq L_1\supsetneq \cdots \supsetneq L_i\supsetneq \pi L_0,
\end{equation}
and each simplex is uniquely determined by its vertices. The
simplicial complex $\cB$ is called the \textit{Bruhat-Tits building}
of $\PGL_n(K)$. Each $L_i/\pi L_0$ is a module over $\cO/\pi\cO\cong
\F_q$, so from (\ref{eq-simB}) we get a strictly decreasing chain of
linear subspaces
$$
L_0/\pi L_0 \supset L_1/\pi L_0\supset \cdots \supset L_i/\pi
L_0\supset 0.
$$
Since $L_0/\pi L_0\cong \F_q^n$ is $n$-dimensional over $\F_q$,
$\cB$ is an infinite $(n-1)$-dimensional simplicial complex.
$\GL_n(K)$ acts on $\cB$ via its natural action on the lattices
(note that $\GL_n(K)$ preserves inclusions of lattices).

\begin{defn} Let $L\in \cL$ be spanned over $\cO$ by the vectors $e_1, \dots, e_n$
in $K^n$. Let $\det(L)$ be the determinant of the matrix having as
its columns the elements $e_1,\dots, e_n$. The \textit{type} of $L$
is the element of $\Z/n\Z$ defined by
$$
\Type(L):=\ord(\det(L))\ \mod\ n.
$$
Note that $\Type(L)=\Type(xL)$ for any $x\in K^\times$, so we can
associate types to the vertices of $\cB$. It is easy to check that
the vertices of any simplex in $\cB$ have distinct types. The action
of $\GL_n(K)$ on $\cB$ does not preserve the types of vertices (in
fact, $\GL_n(K)$ acts transitively on the vertices of $\cB$). A
matrix $g\in \GL_n(K)$ preserves the types of vertices if
$n|\ord(\det(g))$.
\end{defn}

\begin{notn}
Let $z$ be a parameter. Set $[0]_z=1$. For $m\geq 1$, let
$$
[m]_z:=(z^m-1)(z^{m-1}-1)\cdots (z-1).
$$
\end{notn}

\begin{defn}
An \textit{ordered partition of $n$} is an expression of $n$ as an
ordered sum of positive integers. We will write ordered partitions
as row vectors:
$$
\bp = (p_1,\dots, p_h), \quad p_1,\dots, p_h\geq 1,\quad n = p_1+
\cdots + p_h.
$$
Define the length of $\bp=(p_1,\dots, p_h)$ to be $\ell(\bp):=h$.
The set of all ordered partitions of $n$ will be denoted by
$\mathrm{Par}(n)$. For $\bp=(p_1,\dots, p_h)\in \Par(n)$, let
$$
\qbinom{n}{\bp}:=\frac{[n]_q}{[p_1]_q[p_2]_q\cdots [p_h]_q}.
$$
It is obvious that $\qbinom{n}{\bp}$ does not depend on the ordering
of the entries of $\bp$.
\end{defn}

\begin{lem}\label{lem3.3} Let $v\in \Ver(\cB)$. Then for $1\leq i\leq n-1$
$$
\deg_\cB^i(v)=\sum_{\substack{\bp\in \Par(n)\\ \ell(\bp)=i+1}}
\qbinom{n}{\bp}.
$$
\end{lem}
\begin{proof} Suppose $v$ corresponds to the class of the lattice $L$.
Let $V:=L/\pi L\cong \F_q^{n}$. An \textit{$i$-flag} in $V$ is a
chain of vector subspaces
\begin{equation}\label{eq-flag}
V \neq \cF_1 \supsetneq \cF_{2}\supsetneq \cdots\supsetneq \cF_i\
\neq 0.
\end{equation}
From the definition of $\cB$ it is easy to see that the
$i$-simplices of $\cB$ having $v$ as a vertex are in bijection with
the $i$-flags in $V$. Next, to each $i$-flag we associate an ordered
partition of length $i+1$ as follows. Let $d_i:=\dim_{\F_q}\cF_i$.
Then to (\ref{eq-flag}) we associate
$$
(n-d_1, d_1-d_2, \dots, d_{i-1}-d_{i}, d_i).
$$
Denote the number of $k$-dimensional subspaces in $\F_q^m$ by
$\qbinom{m}{k}$. The number of distinct $i$-flags which map to
$\bp=(p_1,\dots, p_{i+1})$ is equal to
$$
f(\bp):=\qbinom{n}{n-p_1}\qbinom{n-p_1}{n-p_1-p_2}\qbinom{n-p_1-p_2}{n-p_1-p_2-p_3}\cdots
\qbinom{n-p_1-\cdots-p_i}{0}.
$$
It is well-known that $\qbinom{m}{k}=\frac{[m]_q}{[k]_q[m-k]_q}$, so
$f(\bp)=\qbinom{n}{\bp}$. Hence
$$
\deg_\cB^i(v)=\sum_{\substack{\bp\in \mathrm{Par}(n)\\
\ell(\bp)=i+1}} f(\bp)=\sum_{\substack{\bp\in \Par(n)\\
\ell(\bp)=i+1}} \qbinom{n}{\bp}.
$$
\end{proof}

\begin{example} Let $n=3$. The length-$2$ ordered partitions of $3$ are $(1,2)$ and
$(2,1)$, so
$$
\deg_\cB^1(v)=\qbinom{3}{(1,2)}+\qbinom{3}{(2,1)}=2\frac{(q^3-1)(q^2-1)(q-1)}{(q^2-1)(q-1)(q-1)}=2(q^2+q+1).
$$
Similarly,
$$
\deg_\cB^2(v)=\qbinom{3}{(1,1,1)}=\frac{(q^3-1)(q^2-1)(q-1)}{(q-1)^3}=(q^2+q+1)(q+1).
$$
\end{example}

We will need the next lemma in $\S$\ref{Sec3}.

\begin{lem}\label{lemAndrews} For $n\geq 1$, we have
$$
\sum_{\bp\in \Par(n)}
(-1)^{\ell(\bp)}\frac{1}{\ell(\bp)}\qbinom{n}{\bp} =
(-1)^{n}\frac{1}{n} [n-1]_q.
$$
\end{lem}
\begin{proof} In this proof we treat $q$ as a formal parameter and
manipulate infinite series and products ignoring the issues of
convergence (for a justification see \cite{Andrews}). We need to
prove the following:
$$
\sum_{h=1}^\infty \frac{(-1)^h}{h}\sum_{\substack{p_1+\cdots+p_h=n \\
p_1\geq 1,\dots ,p_h\geq 1}}\frac{(-1)^n}{[p_1]_q\cdots [p_h]_q} =
\frac{1}{n(q^n-1)}.
$$
We put the left hand-side into a generating series
$$
\sum_{n=1}^\infty x^n \sum_{h=1}^\infty \frac{(-1)^h}{h}\sum_{\substack{p_1+\cdots+p_h=n \\
p_1\geq 1,\dots ,p_h\geq 1}}\frac{(-1)^n}{[p_1]_q\cdots [p_h]_q}
$$
$$
=\sum_{h=1}^\infty \frac{(-1)^h}{h}\sum_{p_1\geq 1,\dots ,p_h\geq
1}\frac{(-x)^{p_1+\cdots +p_h}}{[p_1]_q\cdots [p_h]_q}
$$
$$
= \sum_{h=1}^\infty \frac{(-1)^h}{h} \left(\sum_{m=1}^\infty
\frac{(-x)^m}{[m]_q}\right)^h= -\ln(1+\fE),
$$
where $\fE:=\sum_{m=1}^\infty (-x)^m/[m]_q$. By a formula of Euler
\cite[Cor. 2.2]{Andrews}
$$
\fE=-1+\prod_{i=0}^\infty(1-xq^i)^{-1},
$$
so we have
$$
-\ln(1+\fE) = \sum_{i=0}^{\infty}\ln(1-xq^i)=-\sum_{i=0}^\infty
\sum_{n=1}^\infty \frac{(xq^i)^n}{n}
$$
$$
= -\sum_{n=1}^\infty \frac{x^n}{n}\sum_{i=0}^\infty q^{in} =
\sum_{n=1}^\infty \frac{x^n}{n(q^n-1)}.
$$
We conclude that the coefficient of $x^n$ in the initial generating
series is equal to $1/n(q^n-1)$, which finishes the proof.
\end{proof}

\subsection{Euler-Poincar\'e measure}\label{Sec2.3} Denote $G:=\PGL_n(K)$. $G$ is
a locally compact unimodular topological group. Let $dg$ be the Haar
measure on $\GL_n(K)$ normalized by $\Vol(\GL_n(\cO), dg)=1$. Let
$dz$ be the Haar measure on $K^\times$ normalized by
$\Vol(\cO^\times, dz)=1$. Let $dh:=dg/dz$ be the quotient measure on
$G=\GL_n(K)/K^\times$. Let $\G$ be a discrete subgroup of $G$ and
assume that $\G\bs G$ is compact. Since $G$ acts on $\cB$ via its
natural action on lattices in $\bar{\cL}$, $\G$ also acts on $\cB$.
Assume that $\G$ preserves the types of vertices of $\cB$. Then $\G$
satisfies (\ref{eq-star}), since the types of vertices of any
simplex are distinct. The quotient simplicial complex $\G\bs\cB$ is
finite since $\G\bs G$ is compact, cf. \cite[p. 139]{SerreCGD}. Let
$d\delta$ be the counting measure on $\G$, and $dh/d\delta$ be the
quotient measure on $\G\bs G$. The stabilizer $\G_t$ of $t\in \cB$
is finite since the stabilizer of $t$ in $G$ is bounded and $\G$ is
discrete in $G$; cf. \cite[p. 115]{SerreCGD}. The order $\#\G_s$
does not depend on the choice of $s$ in the orbit $\G\cdot s$.
Therefore, for $s\in \G\bs\cB$, we can define $\# \G_s$ as $\# \G_t$
for some preimage $t$ of $s$ in $\cB$.

\begin{thm}\label{thmEP} Assume $\G$ has a normal torsion-free subgroup of finite
index. Then
\begin{align*}
\chi(\G\bs\cB)=& \frac{1}{n}(-1)^{n-1}[n-1]_q\Vol\left(\G\bs G,
\frac{dh}{d\delta}\right)\\ &+\sum_{i=0}^{n-1} (-1)^i\sum_{s\in
S_i(\G\bs\cB)} \left(1-\frac{1}{\# \G_{s}}\right).
\end{align*}
\end{thm}
\begin{proof} $G$ has an Euler-Poincar\'e measure $\mu$ in the sense of
\cite{SerreCGD} (see page 140 in \textit{loc. cit.}), in fact $\mu$
is necessarily unique. Let $\G'\lhd \G$ be a normal torsion-free
subgroup of finite index. Then $\G'$ is a cocompact subgroup of $G$,
and hence is of type (FL) in the terminology of \cite{SerreCGD} (see
Theorem 3 on page 121 of \textit{loc. cit.}) Therefore, by the
definition of $\mu$, $\chi(\G')=\Vol(\G'\bs G, \mu/d\delta)$. Since
the geometric realization of $\cB$ is contractible,
$\chi(\G')=\chi(\G'\bs\cB)$ (see Proposition 9 on page 91 in
\textit{loc. cit.}) We conclude that $\chi(\G'\bs\cB)=\Vol(\G'\bs G,
\mu/d\delta)$. Since $\mu$ is a Haar measure, it is proportional to
any other Haar measure $\mu'$ on $G$, i.e., $\mu=c\cdot \mu'$ for
some non-zero constant $c\in \R$. A method for computing this
constant is given on page 140 of \textit{loc. cit.} For the measure
$dh$, Theorem 7 on page 150 of \textit{loc. cit.} gives
$$
\mu = \frac{1}{n}(-1)^{n-1}[n-1]_q dh;
$$
see also \cite[Prop 5.3.9]{LaumonCDV}. Thus,
$$
\chi(\G'\bs\cB)= \frac{1}{n}(-1)^{n-1}[n-1]_q\Vol\left(\G'\bs G,
\frac{dh}{d\delta}\right).
$$
Let $H:=\G'\bs \G$. Then $H$ acts on $\G'\bs\cB$ and
$\G\bs\cB=H\bs(\G'\bs\cB)$. Let $s\in S_i(\G\bs\cB)$. Since $\G'$ is
torsion-free, it is easy to see that $\# H_s = \# \G_s$. Hence by
Lemma \ref{lem2.1}
\begin{align*}
\chi(\G\bs\cB)=& \frac{1}{n}(-1)^{n-1}[n-1]_q\frac{1}{\# H}
\Vol\left(\G'\bs G, \frac{dh}{d\delta}\right)\\ &+\sum_{i=0}^{n-1}
(-1)^i\sum_{s\in S_i(\G\bs\cB)} \left(1-\frac{1}{\# \G_{s}}\right).
\end{align*}
Finally, $\frac{1}{\# H} \Vol\left(\G'\bs G,
\frac{dh}{d\delta}\right)=\Vol\left(\G\bs G,
\frac{dh}{d\delta}\right)$.
\end{proof}


\section{Quotients by arithmetic groups}\label{Sec3}

Let $D$ be a central division algebra over $F$ of dimension $n^2$.
For $x\in |F|$, let $D_x:=D\otimes_F F_x$ and $\inv_x(D)\in \Q/\Z$
be the local invariant of $D$ at $x$; see \cite[Ch. 8]{Reiner} for
the definition. Let $R\subset |F|$ be the set of places ramified in
$D$, i.e., the set of places for which $\inv_x(D)\neq 0$. The
following facts can be found in \cite[$\S$32]{Reiner}:
\begin{enumerate}
\item For any $x\in |F|$ there exists $m_x|n$ such that $m_x\cdot
\inv_x(D)=0$.
\item $n$ is the smallest positive integer such that $n\cdot
\inv_x(D)=0$ for all $x$.
\item $R$ is a finite set and $\sum_{x\in R}\inv_x(D)=0$.
\item $D$ is uniquely determined by its local invariants.
\end{enumerate}
These properties obviously imply that $\# R\geq 2$. Assume $D_x$ is
a division algebra over $F_x$ for each $x\in R$; this is equivalent
to $\inv_x(D)$ having exponent $n$ in $\Q/\Z$. (Later in the section
we will assume that $n$ is prime which makes this condition
automatic.) Note that this assumption, combined with (3), implies
that $\# R$ is even for even $n$. From now on we assume
$\infty\not\in R$.

Let $\La$ be a maximal $A$-order in $D$. Since $D$ satisfies the
Eichler condition \cite[(34.3)]{Reiner} and $A$ is a principal ideal
domain, $\La$ is unique up to conjugation in $D$; see
\cite[(35.14)]{Reiner}. Let $\G:=\La^\times$ be the subgroup of
units of $\La$. This is the subset of $\La$ consisting of those
elements whose reduced norm is in $\F_q^\times$. Let $D^\times$ be
the multiplicative group of $D$. $\G$ acts on $\cB$ via the
embedding
$$\G\hookrightarrow D^\times\hookrightarrow (D\otimes K)^\times\cong
\GL_n(K).$$ Let $\bg$ be the image of $\G$ in $G:=\PGL_n(K)$. Note
that $\F_q^\times$ is in the center of $\G$ and $\bg\cong
\G/\F_q^\times$. $\bg$ is a discrete cocompact subgroup of $G$.
Moreover, $\G$ preserves the types of vertices of $\cB$ since
$\ord(\det(\gamma))=0$ for any $\gamma\in \G$. Hence we can apply
Theorem \ref{thmEP} to compute the Euler-Poincar\'e characteristic
of $\G\bs\cB=\bg\bs \cB$ ($\F_q^\times$ acts trivially on $\cB$),
but first we need to determine the stabilizers of simplices in
$\G\bs\cB$. This will be done in a series of lemmas. For $x\in
|F|-\infty$, denote $\La_x=\La\otimes_{A}\cO_x$.

\begin{lem}\label{BTfix}
Let $H$ be a finite subgroup of $\G$. Then there exists a vertex
$v\in \Ver(\cB)$ such that $H\subset \G_v$.
\end{lem}
\begin{proof} This is a consequence of the Bruhat-Tits fixed point
theorem; see the theorem on page 161 in \cite{Brown}.
\end{proof}

\begin{lem}
Let $H$ be a finite subgroup of $\G$. Then $H$ is isomorphic to a
subgroup of $\GL_n(\F_q)$ of order coprime to $p$.
\end{lem}
\begin{proof}
First, we show that every element of $H$ has order coprime to $p$.
Let $g\in H$ be of order $m$. Suppose $p|m$. Replacing $g$ by
$g^{m/p}$, we may assume that $g$ has order $p$. Now $g^p=1$ implies
$(g-1)^p=0$ in $D$. Since $g\neq 1$, this leads to a contradiction,
as $D$ is a division algebra. By Cauchy's theorem, $p$ is coprime to
$\# H$. Next, by Lemma \ref{BTfix}, we know that $H\subset \G_v$ for
some vertex $v\in \cB$. The stabilizer of $v$ in $\GL_n(K)$ is
isomorphic to $K^\times \GL_n(\cO)$. Hence $H$ is isomorphic to a
subgroup of $\GL_n(\cO)$. Consider the reduction map $\GL_n(\cO)\to
\GL_n(\F_q)$. It is well-known that the kernel of this homomorphism
contains torsion elements only of order a power of $p$, so $H$ maps
isomorphically to a subgroup of $\GL_n(\F_q)$.
\end{proof}

\begin{lem}\label{lem4.3}\hfill
\begin{enumerate}
\item Every finite subgroup of $\G$ is contained in a maximal finite
subgroup.
\item A maximal finite subgroup of $\G$ is isomorphic to
$\F_{q^d}^\times$ for some $d|n$. Moreover, if $D^\times$ contains a
subgroup isomorphic to $\F_{q^n}^\times$, then every maximal finite
subgroup of $\G$ is isomorphic to $\F_{q^n}^\times$.
\item The stabilizer in $\G$ of a simplex of $\cB$ is isomorphic to $\F_{q^d}^\times$ for some $d|n$.
\end{enumerate}
\end{lem}
\begin{proof}
(1) Each sequence of finite subgroups of $\G$ ordered by inclusion
contains a maximal element since all such subgroups are isomorphic
to subgroups of $\GL_n(\F_q)$.

(2) Let $H$ be a maximal finite subgroup of $\G$. We can consider
$H$ as a finite subgroup of $\La_x^\times$, $x\in |F|-\infty$. In
particular, let $x\in R$. Then by assumption $D_x$ is a central
division algebra over $F_x$, so $\La_x$ is the unique
$\cO_x$-maximal order in $D_x$. The structure of $\La_x$ is
well-known: $\La_x$ has a unique two-sided maximal ideal $\Pi_x$ and
$\La_x/\Pi_x\cong \F_x^{(n)}$, where $\F_x^{(n)}$ denotes the
degree-$n$ extension of $\F_x$; see \cite[Ch. 3]{Reiner}. Hence a
finite subgroup of $\La_x^\times$ is isomorphic to a subgroup of
$(\F_x^{(n)})^\times$. In particular, $H$ is commutative. Since $H$
is finite and commutative, the subring $\F_q[H]$ of $\La$ generated
over $\F_q$ by $H$ is finite. On the other hand, $D$ is a division
algebra. Hence $\F_q[H]$ is a finite field extension $\F$ of $\F_q$.
Obviously, $\F^\times\subset \La^\times=\G$, thus $H=\F^\times$ by
maximality. Now $L:=\F F$ is an $F$-subfield of $D$ and $[L:
F]=[\F:\F_q]$. This implies that $d:=[\F:\F_q]$ divides $n$ (see
\cite[Cor. A.3.4, p. 255]{LaumonCDV}), and $H\cong \F_{q^d}^\times$.

Let $L':=\F_{q^n}F$. Assume $L'$ embeds into $D$. Any two embeddings
of $L$ into $D$ are conjugate; cf. \cite[Cor. A.3.4, p.
255]{LaumonCDV}. Hence any subfield of $D$ isomorphic to $L$ is
contained in a field isomorphic to $L'$, and $L\cap \La\subset
L'\cap \La$. Suppose $\tau\in L'$ generates $\F_{q^n}^\times$. Then
some power of $\tau$ generates $H$, so $\tau\in \G$. This implies
that $H$ is contained in a subgroup in $\G$ isomorphic to
$\F_{q^n}^\times$. By maximality, $H\cong \F_{q^n}^\times$.

(3) The proof of this part is similar to (2). Let $s$ be a simplex
of $\cB$ and $H=\G_s$. The stabilizer of $s$ in $\GL_n(K)$ is
$K^\times \cdot P$ for some parahoric subgroup $P$, so $H=\G\cap P$;
see \cite[VI.5]{Brown}. Let $\F^\times:=\F_q[H]^\times$. As we saw
in (2), $\F$ is a field and $\F^\times\cong \F_{q^d}^\times$ for
some $d|n$. The lattices in $K^n$ forming the flag corresponding to
$s$ are clearly mapped to themselves under the action of $\F$ (as a
subring of $\M_n(K)$). Since the corresponding elements are
invertible, they lie in $P$. Hence $H\subset \F^\times\subset \G\cap
P=H$. This implies that $\F^\times=H$.
\end{proof}

\begin{lem}\label{lemUV}
Let $H\subset \G$ be a finite subgroup isomorphic to
$\F_{q^n}^\times$. Then $H$ fixes a unique vertex of $\cB$.
\end{lem}
\begin{proof}
By Lemma \ref{BTfix}, $H$ fixes a vertex $v$ in $\cB$. Assume there
is another vertex $w\neq v$ fixed by $H$. Let $B$ be the Euclidean
building associated to $\cB$; see \cite[Ch. VI]{Brown} for the
definition. Let $[v,w]$ be the line segment in $B$ joining $v$ and
$w$; this is well-defined by part (4) of the theorem on page 152 in
\cite{Brown}. $H$ fixes $[v,w]$ pointwise since this is a group of
isometries fixing the endpoints. The union of closed simplices
containing $v$ is a neighborhood of $v$ in $B$. Therefore, there is
a unique simplex $s\in \cB$ of minimal positive dimension which
contains $v$ and intersects $[v,w]$. This simplex must be fixed by
$H$. This implies that $H$ is contained in a proper parahoric
subgroup of $\GL_n(K)$, cf. the proof of Lemma \ref{lem4.3}. After
conjugation, we can assume that $H$ is contained in a standard
non-maximal parahoric subgroup $P$ of $\GL_n(\cO)$. Let $\xi$ be a
generator of $\F_{q^n}^\times$. Consider $\xi$ as a matrix in
$\GL_n(\cO)$ and let $f_\xi$ be its characteristic polynomial. Then
$f_\xi$ coincides with the minimal polynomial of $\xi$ over $\F_q$,
so it is irreducible. The image of $\xi$ in $\GL_n(\F_\infty)$ under
the reduction map $\GL_n(\cO)\to \GL_n(\F_\infty)$ has
characteristic polynomial $f_\xi\ (\mod\ \pi_\infty)=f_\xi$, which
is still irreducible since $\F_\infty\cong \F_q$. On the other hand,
the image of $P$ in $\GL_n(\F_\infty)$ is a proper parabolic
subgroup, so the characteristic polynomials of its elements are
reducible. This leads to a contradiction.
\end{proof}

\begin{notn} Let $x\in |F|$ and let $m\geq 1$ be a positive integer.
Denote
$$
\wp(x,m)= \left\{
  \begin{array}{ll}
    0, & \hbox{if $\textrm{gcd}(m,\deg(x))> 1$;} \\
    1, & \hbox{otherwise.}
  \end{array}
\right.
$$
Let $\wp(R, m):=\prod_{x\in R}\wp(x,m)$.
\end{notn}

\begin{lem}\label{lem4.6}
$\F_{q^m}(T)$ embeds into $D$ if and only if $m|n$ and $\wp(R,m)=1$.
\end{lem}
\begin{proof}
Let $K/F$ be a finite field extension. Then $K$ embeds into $D$ as
an $F$-subalgebra if and only if $[K:F]$ divides $n$ and none of the
places in $R$ split in $K$; see \cite[Cor. A.3.4, p.
255]{LaumonCDV}. There are no ramified places in the extension
$\F_{q^m}F/F$. Moreover, a place $x\in |F|$ splits into
$\mathrm{gcd}(\deg(x),m)$ places in $\F_{q^m}(T)$.
\end{proof}

\begin{defn} Let $L:=\F_{q^n}(T)$ and $B:=\F_{q^n}[T]$; $B$ is the
integral closure of $A$ in $L$. Suppose there exists an embedding
$\phi: L\hookrightarrow D$. We say that $\phi$ is an \textit{optimal
embedding with respect to $\La/B$} if $\phi(L)\cap \La= \phi(B)$,
cf. \cite[p. 26]{Vigneras}; for simplicity, we say \textit{$B$ is
optimally embedded in $\La$}. If $\phi$ exist, then it is known that
$B$ is optimally embedded in at least one maximal $A$-order in $D$,
cf. \cite[p. 384]{DvG}. Since all such maximal $A$-orders are
conjugate in $D$, we can assume that $B$ is optimally embedded in
$\La$. Two optimal embeddings $\phi_1$ and $\phi_2$ of $B$ into
$\La$ are said to be \textit{equivalent modulo $\G$} if
$\phi_2=\gamma\phi_1\gamma^{-1}$ for some $\gamma\in \G$. Denote the
number of optimal embeddings of $B$ into $\La$, which are
non-equivalent modulo $\G$, by $m(B)$. Similarly, for $x\in
|F|-\infty$, denote by $m_x(B_x)$ the number of optimal embeddings
of $B_x:=B\otimes_A \cO_x$ into $\La_x$ which are not equivalent
modulo $\La_x^\times$.
\end{defn}

\begin{thm}[Eichler]\label{thm-EchEmb}
$$m(B)=\prod_{x\in |F|-\infty}m_x(B_x).$$
\end{thm}
\begin{proof}
The idelic proof of this fact given in \cite[Thm. 5.11, p.
92]{Vigneras} in the case of quaternion algebras extends directly to
our case.
\end{proof}

\begin{lem}\label{lem-nR}
If $x\not \in R$, then $m_x(B_x)=1$.
\end{lem}
\begin{proof} To prove the lemma it is enough to show that there is a
unique optimal embedding of $B$ into $\M_n(A)$, up to conjugation by
$\GL_n(A)$. Indeed, if $x\not \in R$, then $D_x\cong \M_n(F_x)$ and
$\La_x^\times\cong \GL_n(\cO_x)$. We can apply Theorem
\ref{thm-EchEmb} to $\M_n(F)$ with $\La=\M_n(A)$. If we show that
there is a unique equivalence class of optimal embeddings of $B$
into $\M_n(A)$, then $m(B)=1$, so by Eichler's theorem $m_x(B_x)=1$.

Fix a generator $\xi$ of $\F_{q^n}$ over $\F_q$, i.e.,
$\F_{q^n}=\F_q[\xi]$. The minimal polynomial $f_\xi$ of $\xi$ over
$\F_q$ has degree $n$: $f_\xi(x)=x^n+a_{n-1}x^{n-1}+\cdots+a_n$.
Note that $B=A[\xi]$. To give an optimal embedding of $B$ into
$\M_n(A)$ is equivalent to specifying a matrix $X$ in $\GL_n(A)$
whose minimal polynomial is $f_\xi$. The subring of $\M_n(A)$
generated by $X$ and the scalar matrices is isomorphic to $B$. The
claim becomes the following statement. There exists a matrix in
$\GL_n(A)$ with minimal polynomial $f_\xi$ and any two such matrices
are conjugate in $\GL_n(A)$.

Consider $\F_{q^n}$ as an $n$-dimensional $\F_q$-vector space. The
action of $\xi$ by multiplication induces a linear transformation
whose minimal polynomial is $f_\xi$. This proves the existence of
the desired matrix $X$, since $\GL_n(\F_q)\subset \GL_n(A)$. The
uniqueness of the conjugacy class of $X$ in $\GL_n(A)$ follows from
the main theorem of \cite{LM}, since $B$ is a principal ideal
domain.
\end{proof}

\begin{lem}\label{lem-R}
If $x\in R$, then $m_x(B_x)=\wp(x,n)\cdot n$.
\end{lem}
\begin{proof}
If $x\in R$, then $D_x$ is a division algebra by assumption. If
$\wp(x, n)=0$, then $L_x:=L\otimes_F F_x$ is not a field hence
cannot be embedded into $D_x$. Now assume $\wp(x, n)=1$, so that
there exists an embedding $L_x\hookrightarrow D_x$. Let
$\inv_x(D)=d/n$ and let $\xi$ be as in the proof of Lemma
\ref{lem-nR}. Then
$$
\La_x=B_x\oplus B_x\tau\oplus\cdots\oplus B_x\tau^{n-1},
$$
where $\tau^n=\varpi_x^d$ (here $\varpi_x$ is a fixed uniformizer of
$\cO_x$), $\tau a =a\tau$ for $a\in \cO_x$ and $\tau
\xi=\xi^{q}\tau$; see (A.2.6) on page 253 in \cite{LaumonCDV}. The
element $\tau$ generates a two-sided ideal $(\tau)$. Any element in
this ideal has reduced norm divisible by $\pi_x$, so cannot be
invertible. We conclude that $\La_x^\times= B_x^\times$. To give an
optimal embedding of $B_x$ into $\La_x$ is equivalent to specifying
an element of $\La_x$ with minimal polynomial $f_\xi$. But these
elements in $\La_x$ are exactly $\xi, \xi^q, \dots, \xi^{q^{n-1}}$.
These elements are conjugate in $\La_x$ only by elements in
$(\tau)$, which is not in $\La_x^\times$. Hence there are $n$
distinct optimal embeddings.
\end{proof}

\begin{prop}\label{propCojGrp}
The number of conjugacy classes of subgroups of $\G$ isomorphic to
$\F_{q^n}^\times$ is equal to $\wp(R,n) n^{\# R-1}$.
\end{prop}
\begin{proof} We keep the notation of the proof of Lemma \ref{lem-nR}. Giving
an optimal embedding of $B$ into $\La$ is equivalent to specifying
an element in $\La$ with minimal polynomial $f_\xi$. Hence the
number of optimal embeddings of $B$ into $\La$ up to conjugation by
$\G$ is equal to the number of conjugacy classes of elements in $\G$
with minimal polynomial $f_\xi$. By Theorem \ref{thm-EchEmb} and
Lemmas \ref{lem-nR}, \ref{lem-R}, the number of such conjugacy
classes is equal to $\wp(R,n)n^{\# R}$. The cyclic subgroup
generated in $\G$ by an element with minimal polynomial $f_\xi$ is
finite and isomorphic to $\F_{q^n}^\times$. Conversely, in a
subgroup of $\G$ isomorphic to $\F_{q^n}^\times$ we can find exactly
$n$ elements with minimal polynomial $f_\xi$ (over $F$). Hence to
prove the proposition it remains to show that if $\gamma_1\neq
\gamma_2$ satisfy $f_\xi(\gamma_i)=0$ and generate the same subgroup
in $\G$, then they are not $\G$-conjugate. Suppose $\gamma_2=\gamma
\gamma_1\gamma^{-1}$ with $\gamma\in \G$. Then $\gamma_1$ and
$\gamma_2$ are also conjugate in $\La_x^\times$ for $x\in R$. But as
we saw in the proof of Lemma \ref{lem-R} this is not the case.
\end{proof}

\begin{prop}\label{prop4.12}
The number of $\G$-orbits of vertices of $\cB$ with stabilizers
isomorphic to $\F_{q^n}^\times$ is equal to $\wp(R, n)n^{\# R-1}$.
\end{prop}
\begin{proof} Using Proposition \ref{propCojGrp}, it is enough to
show that there is a one-to-one correspondence between the set $S$
of $\G$-orbits in $\Ver(\cB)$ with stabilizers isomorphic to
$\F_{q^n}^\times$ and the set $S'$ of conjugacy classes of subgroups
of $\G$ isomorphic to $\F_{q^n}^\times$.

Since the vertices in the same $\G$-orbit have $\G$-conjugate
stabilizers, the map $S\to S'$ given by $v\mapsto \G_v$ is
well-defined. This map is surjective by Lemma \ref{BTfix}. If the
map is not injective, then there exist $v$ and $w$, which are not in
the same $\G$-orbit but $\G_w=\gamma \G_v \gamma^{-1}$. Now $\gamma
v\neq w$, but their stabilizers are equal in $\G$. In particular a
subgroup of $\G$ isomorphic to $\F_{q^n}^\times$ fixes two distinct
vertices in $\cB$, which contradicts Lemma \ref{lemUV}.
\end{proof}

\begin{thm}\label{thmChi} Assume $n$ is prime.
\begin{enumerate}
\item If $s\in S_i(\cB)$ with $i\geq 1$, then $\G_s=\F_q^\times$.
\item For a vertex $v\in \Ver(\cB)$, the stabilizer $\G_v$ is either
$\F_q^\times$ or is isomorphic to $\F_{q^n}^\times$. The number of
$\G$-orbits of vertices of $\cB$ with $\G_v\cong \F_{q^n}^\times$ is
equal to $\wp(R, n) n^{\# R-1}$.
\item
$$
\chi(\G\bs\cB)= \frac{(q-1)(-1)^{n-1}}{[n]_q}\prod_{x\in
R}[n-1]_{q_x} +\wp(R,n)\cdot n^{\#
R-1}\left(1-\frac{q-1}{q^n-1}\right).
$$
\end{enumerate}
\end{thm}
\begin{proof} Let $s$ be an $i$-simplex of $\cB$. By Lemma \ref{lem4.3}, $\G_s\cong \F_{q^d}^\times$ for some
$d|n$. If $n$ is prime, then $d=1$ or $d=n$. Suppose $i>0$. Since
$\G_s$ fixes all the vertices of $s$, $\G_s$ fixes at least two
distinct vertices of $\cB$. By Lemma \ref{lemUV}, $\G_s\not \cong
\F_{q^n}^\times$, so $\G_s=\F_q^\times$. On the other hand, by
Proposition \ref{prop4.12}, the number of $\G$-orbits of vertices of
$\cB$ with stabilizers isomorphic to $\F_{q^n}^\times$ is equal to
$\wp(R, n)n^{\# R-1}$. This proves (1) and (2).

By Proposition 4.1 and (6.9) in \cite{PapCrelle},
\begin{align}\label{myVolume}
\Vol\left(\bg\bs G,
\frac{dh}{d\delta}\right)&=n\frac{(q-1)(-1)^{\#R\cdot
(n-1)}}{[n-1]_q[n]_q}\prod_{x\in R}[n-1]_{q_x}\\ \nonumber &=
\frac{n(q-1)}{[n-1]_q[n]_q}\prod_{x\in R}[n-1]_{q_x}.
\end{align}
(The last equality follows from the fact that $\# R\cdot (n-1)$ is
even for prime $n$.)

$\bg$ has normal torsion-free subgroups of finite index; for
example, the image in $G$ of a principal congruence subgroup of $\G$
is such a subgroup. Now the formula of (3) follows from Theorem
\ref{thmEP}.
\end{proof}

\begin{cor}\label{corCong} If $n$ is prime, then
$\chi(\G\bs\cB)\equiv 1\ (\mod\ q)$.
\end{cor}
\begin{proof} From the formula for $\chi(\G\bs\cB)$ in Theorem
\ref{thmChi}, it is easy to see that $\chi(\G\bs\cB)\equiv (-1)^{\#
R\cdot (n-1)}\ (\mod\ q)$. On the other hand, $\# R\cdot (n-1)$ is
even.
\end{proof}

\begin{example}
Let $n=3$ and $R=\{x,y\}$ with $\deg(x)=\deg(y)=1$. Then
$$
\chi(\G\bs\cB)=\frac{(q-1)^3(q^2-1)^2}{(q-1)(q^2-1)(q^3-1)}+3\left(1-\frac{q-1}{q^3-1}\right)=q+1.
$$

Let $n=3$ and $R=\{x,y\}$ with $\deg(x)=1$ and $\deg(y)=3$. Then
$$
\chi(\G\bs\cB)=\frac{(q-1)^2(q^2-1)(q^3-1)(q^6-1)}{(q-1)(q^2-1)(q^3-1)}=(q-1)(q^6-1).
$$
\end{example}

\begin{rem}\label{rem3.16} The first equality in (\ref{myVolume}) is proven in
\cite{PapCrelle} for arbitrary $n$, assuming $D_x$ is a division
algebra for all $x\in R$. Note that this assumption also implies the
second equality in (\ref{myVolume}) as then $\#R\cdot (n-1)$ is even
(see the beginning of this section). If $\wp(R, m)=0$ for every
$m>1$ dividing $n$, then by Lemma \ref{lem4.6} the only torsion
elements in $\G$ are the elements of the center $\F_q^\times$. Now
one can apply Theorem \ref{thmEP} to conclude that $\chi(\G\bs \cB)$
is given by the formula in Theorem \ref{thmChi}. On the other hand,
if we do not assume $\wp(R, m)=0$ for every $m>1$ dividing $n$, then
the assumption on $n$ being prime cannot be omitted from Theorem
\ref{thmChi}. For example, if we take $n=4$ and $R=\{x,y\}$ with
$\deg(x)=\deg(y)=1$, then the formula for $\chi(\G\bs \cB)$ in
Theorem \ref{thmChi} generally gives non-integer values.
\end{rem}

\begin{rem} It is not completely obvious that the formula for $\chi(\G\bs\cB)$ in
Theorem \ref{thmChi} always produces integer values. As we indicated
in the previous remark, the requirement on $n$ being prime is
necessary to make this happen. To see that the formula produces
integer values, one can argue as follows:

First, suppose $n$ divides $\deg(y)$ for some $y\in R$, i.e.,
$\wp(R,n)=0$. Fix some $z\in R$, $z\neq y$. We have
$[n]_q=(q^n-1)[n-1]_q$. Now $(q^n-1)$ divides $(q_y-1)$ and
$[n-1]_q$ divides $[n-1]_{q_z}$ since $q|q_z$. Therefore, $[n]_q$
divides $[n-1]_{q_y}[n-1]_{q_z}$.

Next, suppose $n$ is coprime to $\deg(x)$ for all $x\in R$, i.e.,
$\wp(R,n)=1$. We treat $q$ as a parameter and write
$q^n-1=(q-1)\Phi(q)$. The formula in Theorem \ref{thmChi} becomes
$$
\frac{(-1)^{n-1}\prod_{x\in R}[n-1]_{q_x}+[n-1]_qn^{\#
R-1}(\Phi(q)-1)}{\Phi(q)[n-1]_q}.
$$
If $n$ is prime, then $\Phi(q)=q^{n-1}+q^{n-2}+\cdots +q+1$ is
coprime to $[n-1]_q$ (as polynomials in $\C[q]$). Since $[n-1]_q$
obviously divides the numerator, it is enough to show that $\Phi(q)$
divides the numerator. $\Phi(q)=(q-\zeta_1)\cdots (q-\zeta_{n-1})$
has degree $n-1$ and its zeros $\{\zeta_1,\dots, \zeta_{n-1}\}$ are
the primitive $n$th roots of $1$. It is enough to show that any
$\zeta\in \{\zeta_1,\dots, \zeta_{n-1}\}$ is a zero of the
numerator. Note that $\{\zeta, \dots, \zeta^{n-1}\}=\{\zeta_1,\dots,
\zeta_{n-1}\}$. Let $x\in R$ and $m:=\deg(x)$. Consider
$$
f(q)=[n-1]_{q_x}=(q^{m(n-1)}-1)(q^{m(n-2)})\cdots (q^m-1).
$$
Since $m$ is coprime to $n$,
\begin{align*}
f(\zeta)&=(\zeta^{m(n-1)}-1)(\zeta^{m(n-2)}-1)\cdots (\zeta^m-1)\\
&=\prod_{i=1}^{n-1}(\zeta_i-1)=(-1)^{n-1}\Phi(1)=(-1)^{n-1}n.
\end{align*}
Hence the numerator with $q=\zeta$ is equal to
$$
(-1)^{(n-1)(\#R+1)}n^{\#R}-(-1)^{n-1}n^{\# R}=0.
$$
\end{rem}

\begin{rem}
Let $n$ be arbitrary and $\bg'$ be a normal, finite index,
torsion-free subgroup of $\bg$. Since $\bg'$ is a discrete,
cocompact, torsion-free subgroup of $G$, by a result of Garland
\cite{Garland} and Casselman \cite{Casselman}, $H^i(\bg', \Q)\cong
H^i(\bg'\bs \cB, \Q)=0$ for $1\leq i\leq n-2$. There is a spectral
sequence $H^i(\bg/\bg', H^j(\bg'\bs\cB, \Q))\Rightarrow
H^{i+j}(\bg\bs \cB, \Q)$. Since $\bg/\bg'$ is finite, $H^i(\bg/\bg',
\Q)=0$ for $i\geq 1$. We conclude that
$$
H^i(\G\bs\cB, \Q)=0 \text{ for } 1\leq i\leq n-2.
$$
Since $\dim_\Q H^0(\G\bs\cB, \Q)=1$, $$\dim_\Q H^{n-1}(\G\bs\cB,
\Q)= (-1)^{n-1}(\chi(\G\bs\cB)-1).$$ Thus, when $n$ is prime,
Theorem \ref{thmChi} provides an explicit expression for the
dimension of this cohomology group. Moreover, Corollary
\ref{corCong} implies that this number is always a multiple of $q$.
\end{rem}

\begin{thm}\label{LastThm} Suppose $n$ is prime, and $R=\{x,y\}$ with
$\deg(x)=\deg(y)=1$. For $0\leq i\leq n-1$, denote by $\theta_i$ the
number of $i$-simplices in $\G\bs\cB$. Then $\theta_0=n$, and for
$1\leq i\leq n-1$
$$
\theta_i = \frac{n}{(i+1)} \frac{(q-1)}{(q^n-1)}
\sum_{\substack{\bp\in \Par(n)\\ \ell(\bp)=i+1}} \qbinom{n}{\bp}.
$$
\end{thm}
\begin{proof} In general, by Theorem \ref{thmChi}, the number of vertices in
$\G\bs\cB$ whose preimages have stabilizers isomorphic to
$\F_{q^n}^\times$ is $\wp(R,n)n^{\# R-1}$; the other
$\theta_0-\wp(R,n)n^{\# R-1}$ vertices have preimages with
stabilizers isomorphic to $\F_q^\times$. Moreover,
$\G_v/\F_q^\times$ acts freely on the simplices containing $v$.
Denoting $\deg_\cB^i=\deg_\cB^i(v)$ for a vertex $v\in\cB$, we can
express all $\theta_i$, $1\leq i\leq n-1$, as linear functions of
$\theta_0$
\begin{equation}\label{eqLinRel}
\theta_i=\frac{\deg^{i}_\cB}{i+1}\left(\frac{(q-1)\wp(R,n)n^{\#R
-1}}{q^n-1}+(\theta_0-\wp(R,n)n^{\#R -1})\right);
\end{equation}
we count the number of $i$-simplices in $\G\bs\cB$ containing a
given vertex, add these numbers over all vertices, and then divide
by $(i+1)$ since every $i$-simplex has exactly $(i+1)$ vertices. On
the other hand, $\chi(\G\bs\cB)=\sum_{i=0}^{n-1}(-1)^i \theta_i$. If
we substitute into this formula the expressions in (\ref{eqLinRel}),
then we get a linear relation between $\theta_0$ and $\chi$. Since
Lemma \ref{lem3.3} gives a formula for the generalized degrees
$\deg^{i}_\cB$ and Theorem \ref{thmChi} gives a formula for
$\chi(\G\bs\cB)$, one can always compute the numbers $\theta_i$ in
each specific case.

From the previous discussion, it is clear that to prove the theorem
it is enough to show $\theta_0=n$. When $R=\{x,y\}$ and
$\deg(x)=\deg(y)=1$,
$$
\chi(\G\bs\cB) =\frac{(q-1)}{(q^n-1)}((-1)^{n-1}[n-1]_q-n)+n.
$$
Now, using the linear relation between $\chi(\G\bs\cB)$ and
$\theta_0$, one easily checks that $\theta_0=n$ is equivalent to the
identity
$$
1+\sum_{i=1}^{n-1}(-1)^i \frac{1}{i+1}\sum_{\substack{\bp\in
\Par(n)\\ \ell(\bp)=i+1}} \qbinom{n}{\bp} = (-1)^{n-1}\frac{1}{n}
[n-1]_q,
$$
which is the statement of Lemma \ref{lemAndrews}.
\end{proof}

\begin{rem} Theorem \ref{LastThm} says that in the case when $R=\{x,y\}$
and $\deg(x)=\deg(y)=1$ one can visualize $\G\bs\cB$ as
$$\theta_{n-1}= \prod_{i=1}^{n-1}\frac{q^i-1}{q-1}$$
$(n-1)$-simplices glued to each other at their vertices and along
some of their higher dimensional faces. How exactly the higher
dimensional faces are glued is easy to describe for small $n$, but
becomes complicated as $n$ grows.

Suppose $n=2$. Then $\theta_0=2$, $\theta_1=1$. Hence $\G\bs\cB$ is
a segment (= two vertices joined by an edge).

Suppose $n=3$. Then $\theta_0=\theta_1=3$ and $\theta_2=q+1$. Hence
$\G\bs\cB$ consists of $(q+1)$ triangles glued together along their
boundaries.
\end{rem}



\begin{thebibliography}{10}

\bibitem{Andrews}
G.~Andrews, \emph{The theory of partitions}, Cambridge Univ. Press,
1984.

\bibitem{Brown}
K.~Brown, \emph{Buildings}, Springer, 1989.

\bibitem{Casselman}
W.~Casselman, \emph{On a $p$-adic vanishing theorem of {Garland}},
Bull. Amer. Math. Soc. \textbf{80} (1974), 1001--1004.

\bibitem{DvG}
M.~Denert and J.~Van Geel, \emph{The class number of hereditary
orders in non-{Eichler} algebras over global fuction fields}, Math.
Ann. \textbf{282} (1988), 379--393.

\bibitem{Garland}
H.~Garland, \emph{$p$-adic curvature and the cohomology of discrete
groups}, Ann. Math. \textbf{97} (1973), 375--423.

\bibitem{GekelerKF}
E.-U. Gekeler, \emph{Automorphe {Formen} \"uber {$\mathbb{F}_q(T)$}
mit kleinem {F\"uhrer}}, Abh. Math. Sem. Univ. Hamburg \textbf{55}
(1985), 111--146.

\bibitem{GN}
E.-U. Gekeler and U.~Nonnengardt, \emph{Fundamental domains of some
arithmetic groups over function fields}, Internat. J. Math.
\textbf{6} (1995), 689--708.

\bibitem{GR}
E.-U. Gekeler and M.~Reversat, \emph{Jacobians of {Drinfeld} modular
curves}, J. Reine Angew. Math. \textbf{476} (1996), 27--93.

\bibitem{LM}
C.~Latimer and C.~MacDuffee, \emph{A correspondence between classes
of ideals and classes of matrices}, Ann. Math. \textbf{34} (1933),
313--316.

\bibitem{LaumonCDV}
G.~Laumon, \emph{Cohomology of {Drinfeld} modular varieties: {Part
I}}, Cambridge Univ. Press, 1996.

\bibitem{Munkres}
J.~Munkres, \emph{Elements of algebraic topology}, Addison-Wesley,
1984.

\bibitem{PapGenus}
M.~Papikian, \emph{Genus formula for modular curves of
{$\cD$}-elliptic sheaves}, Arch. Math. \textbf{92} (2009), 237--250.

\bibitem{PapCrelle}
M.~Papikian, \emph{Modular varieties of $\mathcal{D}$-elliptic
sheaves and the {Weil-Deligne} bound}, J. Reine Angew. Math.
\textbf{626} (2009), 115--134.

\bibitem{Reiner}
I.~Reiner, \emph{Maximal orders}, Academic Press, 1975.

\bibitem{SerreCGD}
J.-P. Serre, \emph{Cohomologie des groupes discrets}, Ann. Math.
Studies \textbf{70} (1970), 77--169.

\bibitem{Soule}
C.~Soul\'e, \emph{Chevalley groups over polynomial rings}, LMS Lect.
Notes \textbf{36} (1979), 359--367.

\bibitem{Vigneras}
M.-F. Vign\'eras, \emph{Arithm\'etiques des alg\`ebres de
quaternions}, LNM 800, 1980.

\end{thebibliography}

\end{document}